\documentclass[12pt]{article}
\usepackage[T2A]{fontenc}
\usepackage[cp1251]{inputenc}
\usepackage[english]{babel}
\usepackage{amssymb,amsmath,amsthm,amsfonts}
\usepackage{color}
\usepackage{cite}
\usepackage[a4paper,left=20mm,right=20mm,bottom=20mm,top=20mm]{geometry}

 \newtheorem{thm}{Theorem}[section]
 \newtheorem{cor}[thm]{Corollary}
 \newtheorem{lem}[thm]{Lemma}
 
 \theoremstyle{definition}
 \newtheorem{defn}[thm]{Definition}
 \theoremstyle{remark}
 
 \newtheorem*{ex}{Example}

\begin{document}

\title{On the supersolubility of a finite group with NS-supplemented Sylow subgroups}
\author{V.\,S.~Monakhov, A.\,A.~Trofimuk}


\maketitle

\begin{abstract}
A subgroup $A$ of a group~$G$ is said to be {\sl NS-supplemented} in $G$, if
there exists a subgroup~$B$ of $G$ such that $G=AB$ and whenever $X$~is a normal subgroup of~$A$ and $p\in \pi(B)$, there exists a Sylow  $p$-subgroup~$B_p$ of~$B$ such that $XB_p=B_pX$.
In this paper, we proved the supersolubility of a group with  NS-supplemented non-cyclic Sylow subgroups.
The solubility of a group with NS-supplemented maximal subgroups is obtained.
\end{abstract}

\section{Introduction}

Throughout this paper, all groups are finite and $G$ always denotes a finite group. We use the standard notations and terminology of~\cite{hup}. The set of all prime divisors of the order of $G$ is denoted by~$\pi (G)$. The notation $Y\le X $ means that $Y$ is a subgroup of a group~$X$. The  semidirect product of a normal subgroup $A$ and a subgroup $B$ is denoted by $[A]B$.

By the Zassenhaus Theorem (\cite[IV.2.11]{hup}),  a group $G$ with cyclic Sylow subgroups has a cyclic Hall subgroup $H$ such that the quotient $G/H$ is also cyclic.  In particular, $G$ is supersoluble.

A group $G$ with abelian Sylow subgroups may be non-soluble (for example, $PSL(2,5)$)
and the compositional factors of $G$ are known~\cite{wal}.

In some papers, the sufficient conditions of solubility and supersolubility of a group in which Sylow subgroups permute with some subgroups  are established. For example,
the supersolubility of a group~$G$ such that every Sylow subgroup~$P$ of $G$ permutes with subgroups of some supplement of~$P$  in~$G$ is obtained in works~\cite{m06}--\cite{guo}.

The following concept is introduced in \cite{ar}.

\begin{defn} \label{d1}
Two subgroups $A$ and $B$ of a group~$G$ are said to be
{\sl NS\nobreakdash-\hspace{0pt}permutable},
if they satisfy the following conditions:

$(1)$
whenever $X$~is a normal subgroup of~$A$ and~$p\in \pi (B)$,
there exists a Sylow $p$-subgroup~$B_p$ of~$B$ such that~$XB_p=B_pX$;

$(2)$
whenever $Y$~is a normal subgroup of~$B$ and~$p\in \pi (A)$,
there exists a Sylow $p$-subgroup~$A_p$ of~$A$ such that~$YA_p=A_pY$.

Moreover, if $G=AB$, we say that~$G$ is an {\sl NS\nobreakdash-\hspace{0pt}permutable
product} of the subgroups~$A$ and~$B$.
\end{defn}

The totally permutable~\cite{bal} and totally c-permutable~\cite{sk} subgroups are
NS\nobreakdash-\hspace{0pt}permutable~\cite[Lemma 2]{ar}.
The supersolubility of a group $G=AB$ which is the
NS\nobreakdash-\hspace{0pt}permutable
product of supersoluble subgroups~$A$ and~$B$ is obtained in \cite{ar}.

We introduce the following

\begin{defn} \label{d2}

A subgroup $A$ of a group~$G$ is said to be {\sl NS-supplemented} in $G$, if
there exists a subgroup~$B$ of $G$ such that:

$(1)$ $G=AB$;

$(2)$ whenever $X$~is a normal subgroup of~$A$ and $p\in \pi(B)$,
there exists a Sylow  $p$-subgroup~$B_p$ of~$B$ such that $XB_p=B_pX$.

In this case we say that~$B$ is a {\sl NS-supplement} of~$A$ in~$G$.
\end{defn}

In this paper, we proved the supersolubility of a group in which every non-cyclic Sylow subgroup is NS-supplemented. The solubility of a group with NS-supplemented maximal subgroups is obtained.

\section{Preliminaries}\label{sec2}

Definition~\ref{d2} implies the following result for~$X=A$.

\begin{lem} \label{l1}
Let  $A$ be an  NS-supplemented subgroup of $G$ and $B$~is its  NS\nobreakdash-\hspace{0pt}supplement in $G$. Then for every $p\in \pi(B)$ there exists a Sylow $p$-subgroup~$B_p$ of $B$ such that~$AB_p=B_pA$.
\end{lem}

\begin{lem} \label{l2}
Let $K$~be a normal subgroup of $G$. If $A$ is  NS-supplemented in $G$ and $B$ is its NS-supplement in $G$, then  $AK/K$ is  NS-supplemented in~$G/K$ and $BK/K$~ is its  NS-supplement in $G/K$.
\end{lem}

\begin{proof}
It's obvious that $G/N=(AK/K)(BK/K)$. Let $X/K$~be a normal subgroup of~$AK/K$ and $p\in \pi(BK/K)$.
Then $X=(A\cap X)K$ and $A\cap X$ is normal in~$A$. By the hypothesis,
for every $p\in \pi (B)$ there exists a Sylow  $p$-subgroup~$B_p$
of~$B$ such that $(A\cap X)B_p=B_p(A\cap X)$. Hence
$$
((A\cap X)K)B_p=B_p((A\cap X)K)
$$
and $X/K$ permutes with Sylow $p$-subgroup~$B_pK/K=(BK/K)_p$
of~$BK/K$.
\end{proof}

\begin{lem} \label{tk} \emph {(\cite[Theorem 2]{tk})}
Let $G$ be a group with $p\in \pi (G)$ and $p\ne 3$.
If $G$ has a  Hall $\{p,r\}$-subgroup for every~$r\in \pi (G)$, then~$G$ is $p$-soluble.
\end{lem}

\begin{lem} \label{gur} \emph {(\cite[Corollary 3]{gur})}
Let  $G$ be a group such that every maximal subgroup has prime power index. Then~$G=S(G)$ or~$G/S(G)\simeq PSL(2,7)$.
 \end{lem}

Here $S(G)$~is the maximal normal soluble subgroup of~$G$.

\section{Groups with NS-supplemented subgroups}\label{sec3}

\begin{thm}\label{t1}
If a Sylow $p$-subgroup $P$ of $G$ is NS-supplemented in $G$, then~$G$ is $p$-supersoluble in each of the following cases:

$(1)$ $p\neq 3$;

$(2)$ $p=3$ and $G$ is 3-soluble.
\end{thm}

\begin{proof}
Let~$B$ be an NS-supplement of~$P$ in~$G$.
By Lemma~\ref{l1}, for any $q\in \pi (B)\setminus \{p\}$ there exists a Sylow $q$-subgroup
$Q$ of~$B$ such that~$PQ=QP$. The subgroup~$PQ$ is a Hall $\{p,q\}$-subgroup of~$G$. Since $q$~is an arbitrary prime of $\pi (G)\setminus \{p\}$, it follows that by Lemma~\ref{tk},~$G$ is $p$-soluble for~$p\neq 3$ and by the hypothesis, $G$ is 3-soluble for~$p=3$.

We use induction on the order of $G$. Let $N$ be an arbitrary non-trivial normal subgroup in $G$.
Then by Lemma~\ref{l2}, a Sylow $p$-subgroup $PN/N$ is NS-supplemented in $G/N$. By induction, $G/N$ is $p$-supersoluble, $O_{p^\prime}(G)=1$ and $N=O_p(G)\ne 1$.

We choose a subgroup $X$ of $G$ such that $X\leq N\cap Z(P)$ and $|X|=p$. Since $X$ is normal in $P$, it follows that for every $r\in \pi(B)$ there exists a Sylow $r$-subgroup $R$ of $B$ such that $XR=RX$. If $p\neq r$, then the subgroup $N\cap XR=X(N\cap R)=X$ is normal in~$XR$. This is true for any prime~$r$,
hence $X$ is normal in $G$. Since the quotient $G/X$ is $p$-supersoluble by induction, $G$ is $p$-supersoluble.
\end{proof}

\begin{cor}
If all non-cyclic Sylow subgroups of $G$ are NS-supplemented in $G$, then~$G$ is supersoluble.
\end{cor}

\begin{proof}
Let $p$ be the smallest prime of $\pi(G)$ and $P$ be a Sylow
$p$-subgroup of $G$. If $P$ is cyclic, then $G$ is
$p$-nilpotent~\cite[IV.2.8]{hup}. If $P$ is non-cyclic,
then $P$ is NS-supplemented in $G$ and by Theorem~\ref{t1}, $G$ is $p$-nilpotent.
In particular,  $G$ is soluble and  we apply Theorem~\ref{t1} for each~$r\in \pi(G)$.
Let $R$ be a Sylow $r$-subgroup of $G$.
If $R$ is cyclic, then $R$ is $r$-supersoluble. If $R$
is non-cyclic, then $R$ is NS-supplemented in $G$ and $G$ is
$r$-supersoluble by Theorem~\ref{t1}. Thus, $G$ is $r$-supersoluble for any $r\in \pi(G)$. Consequently,
$G$ is supersoluble.
\end{proof}

\begin{ex}
The group $PSL(2,7)$ is an NS-supplement of  its Sylow 3-subgroup.  Hence we can not omit the condition  $\ll$group is 3-soluble$\gg$  in Theorem~\ref{t1}.
\end{ex}

\begin{thm}\label{t2}
If all maximal subgroups of  $G$ are  NS-supplemented in $G$,
then $G$ is soluble.
\end{thm}

\begin{proof}
We use induction on the order of $G$. By Lemma~\ref{l2}, all non-trivial quotients are soluble, hence~$S(G)=1$.
Let $M$~be a maximal subgroup of $G$ and $B$~is its NS-supplement in~$G$.
By the hypothesis, for every $p\in \pi(B)$ there exists a Sylow $p$-subgroup $B_p$ of~$B$ such that $B_pM=MB_p$. Since $M\ne G$, there exists
$r\in \pi(B)$ and a Sylow $r$-subgroup  $B_r$ such that $MB_r=G$.  Hence $|G:M|=r^b$. Consequently, every maximal subgroup of $G$ has prime power index. By Lemma~\ref{gur},
$G$ is either soluble, or $G\simeq PSL(2,7)$.
The group $PSL(2,7)$ has  a maximal subgroup  $H\simeq [Z_7]Z_3$ and it has not a subgroup of the order $7\cdot 2^3$. Hence $H$ is not NS-supplemented in $PSL(2,7)$. Consequently,  $G$ is soluble.

\end{proof}

The following example shows that the group that satisfies the hypotheses of Theorem~\ref{t2} can be  non-supersoluble.

\begin{ex}
In the group (\cite{gap}, IdGroup=[72,39])
$$
G=\langle a,b,c\mid a^3=b^3=c^8, \ ab=ba,\ a^c=b,\ b^c=ab\rangle
$$
all maximal subgroups are the following subgroups:
$$
M_1=[(\langle a\rangle\times \langle b\rangle)] \langle c^2\rangle,\
M_i=\langle c^x\rangle, \ x\in \langle a\rangle\times \langle b\rangle.
$$
Moreover, all maximal subgroups are NS-supplemented in~$G$ and the subgroup $M_1$ is non-supersoluble.

\end{ex}


\begin{thebibliography}{1}

\bibitem{hup}
B. Huppert, \textit{Endliche Gruppen I}, Springer, Berlin-Heidelberg-New York, 1967.

\bibitem{wal}
J.\,H. Walter, \textit{Characterization of finite groups with abelian
Sylow 2-subgroups,} Annals of  math. \textbf{89}:3 (1969), 405-514.


\bibitem{m06}
V.\,S. Monakhov,  \textit{Finite groups with seminormal Hall subgroups,}
Mathematical notes  \textbf{80}:4 (2006), 542-549.

\bibitem{guo}
W. Guo,  \textit{Finite groups with seminormal Sylow subgroups,}
Acta Mathematica Sinica \textbf{24}:10 (2008), 1751-1758.

\bibitem{ar}
M. Arroyo-Jorda,  P. Arroyo-Jorda, A. Martinez-Pastor,
M.\,D. Perez-Ramos,  \textit{On finite products of groups and supersolubility,}
J. Algebra \textbf{323} (2010),  2922-2934.

\bibitem{bal}
A. Ballester-Bolinches, R. Esteban-Romero, M. Asaad, \textit{Products of finite groups},
Walter de Gruyter, Berlin-New York, 2010.


\bibitem{sk}
W. Guo, K.\,P. Shum, A.\,N. Skiba,
\textit{Criterions of supersolubility for products of supersoluble groups,}
Publ. Math. Debrecen \textbf{68}:3-4 (2006), 433-449.

\bibitem{tk}
V.\,N. Tyutyanov, V.\,N. Kniahina,
\textit{Finite groups with biprimary Hall subgroups,}
J. Algebra \textbf{443} (2015), 430-440.

\bibitem{gur}
R.\,M. Guralnick,  \textit{Subgroups of prime power index in a simple group,}
J. Algebra  \textbf{81} (1983), 304-311.

\bibitem{gap}
The GAP Group: GAP --- Groups, Algorithms, and Programming.
Ver. GAP 4.9.2 released on 4 July 2018. http://www.gap-system.org.


\end{thebibliography}
\end{document}